\documentclass[
11pt]{amsart}

\usepackage{amsmath}
\usepackage{amssymb}
\usepackage{latexsym}
\usepackage{amsfonts}
\usepackage{graphpap}

\begin{document}

\newcommand{\ci}[1]{_{ {}_{\scriptstyle #1}}}
\newcommand{\ti}[1]{_{\scriptstyle \text{\rm #1}}}

\newcommand{\norm}[1]{\ensuremath{\|#1\|}}
\newcommand{\abs}[1]{\ensuremath{\vert#1\vert}}
\newcommand{\nm}{\,\rule[-.6ex]{.13em}{2.3ex}\,}

\newcommand{\p}{\ensuremath{\partial}}
\newcommand{\pr}{\mathcal{P}}

\newcounter{vremennyj}

\newcommand\cond[1]{\setcounter{vremennyj}{\theenumi}\setcounter{enumi}{#1}\labelenumi\setcounter{enumi}{\thevremennyj}}

\newcommand{\pbar}{\ensuremath{\bar{\partial}}}
\newcommand{\db}{\overline\partial}
\newcommand{\D}{\mathbb{D}}
\newcommand{\T}{\mathbb{T}}
\newcommand{\C}{\mathbb{C}}
\newcommand{\N}{\mathbb{N}}
\newcommand{\bP}{\mathbb{P}}

\newcommand{\bS}{\mathbf{S}}
\newcommand{\bk}{\mathbf{k}}

\newcommand\cE{\mathcal{E}}
\newcommand\cP{\mathcal{P}}
\newcommand\cC{\mathcal{C}}
\newcommand\cH{\mathcal{H}}
\newcommand\cU{\mathcal{U}}
\newcommand\cQ{\mathcal{Q}}

\newcommand{\be}{\mathbf{e}}

\newcommand{\la}{\lambda}
\newcommand{\e}{\varepsilon}

\newcommand{\td}{\widetilde\Delta}

\newcommand{\tto}{\!\!\to\!}
\newcommand{\wt}{\widetilde}
\newcommand{\shto}{\raisebox{.3ex}{$\scriptscriptstyle\rightarrow$}\!}

\newcommand{\La}{\langle }
\newcommand{\Ra}{\rangle }
\newcommand{\ran}{\operatorname{ran}}
\newcommand{\tr}{\operatorname{tr}}
\newcommand{\codim}{\operatorname{codim}}
\newcommand\clos{\operatorname{clos}}
\newcommand{\spn}{\operatorname{span}}
\newcommand{\lin}{\operatorname{Lin}}
\newcommand{\rank}{\operatorname{rank}}
\newcommand{\re}{\operatorname{Re}}
\newcommand{\vf}{\varphi}
\newcommand{\f}{\varphi}


\newcommand{\entrylabel}[1]{\mbox{#1}\hfill}

\newenvironment{entry}
{\begin{list}{X}%
  {\renewcommand{\makelabel}{\entrylabel}%
      \setlength{\labelwidth}{55pt}%
      \setlength{\leftmargin}{\labelwidth}
      \addtolength{\leftmargin}{\labelsep}%
   }%
}%
{\end{list}}



\numberwithin{equation}{section}

\newtheorem{thm}{Theorem}[section]
\newtheorem{lm}[thm]{Lemma}
\newtheorem{cor}[thm]{Corollary}
\newtheorem{prop}[thm]{Proposition}

\theoremstyle{remark}
\newtheorem{rem}[thm]{Remark}
\newtheorem*{rem*}{Remark}

\title[Similarity]{Similarity of operators and geometry of eigenvector bundles}
\author{Hyun-Kyoung Kwon}  \author{Sergei Treil}
\thanks{The work of S.~Treil was supported by the National Science Foundation under Grant  DMS-0501065}
\address{Department of Mathematics\\ Brown University\\ 151 Thayer Street Box 1917\\ Providence, RI USA  02912}


\begin{abstract}
We characterize the contractions that are similar to the backward shift
in the Hardy space $H^2$. This characterization is given in terms of the geometry of  the eigenvector bundles of the operators.
\end{abstract}

\maketitle
\setcounter{tocdepth}{1}
\tableofcontents
\section*{Notation}
\begin{entry}
\item[$:=$] equal by definition;\medskip
\item[$\C$] the complex plane;\medskip
\item[$\D$] the unit disk, $\D:=\{z\in\C:\abs{z}<1\}$;\medskip
\item[$\T$] the unit circle, $\T:=\p\D=\{z\in\C:\abs{z}=1\}$;\medskip
\item[$\La \cdot ,\cdot \Ra_H$] inner product on the space H;\medskip
\item[$\Delta$]Normalized Laplacian, $\Delta = \db \p = \p
\db = \frac14\left(\frac{\p^2}{\p x^2} + \frac{\p^2}{\p y^2} \right)$;\medskip
\item[$\mathfrak{S}_2$]Hilbert-Schmidt  class of operators;\medskip 
\item[$\norm{\cdot}, \nm\cdot \nm$]  norm; since we are dealing with matrix- and operator-valued functions, we will use the symbol $\|\,.\,\|$ (usually with a
subscript) for the norm in a function space, while $\nm\,.\,\nm$
is used for the norm in the underlying vector (operator) space.
 Thus, for a vector-valued function $f$ the symbol $\|f\|_2$ denotes its $L^2$-norm, but the symbol $\nm f\nm$ stands
for the scalar-valued function whose
value at a point $z$ is the norm of the vector $f(z)$;  \medskip

\item[$H^2$, $H^\infty$] Hardy classes of analytic functions,
$$
H^p := \left\{ f\in L^p(\T) : \hat f (k) := \int_\T f(z) z^{-k}
\frac{|dz|}{2\pi} = 0\ \text{for } k<0\right\}.
$$
Hardy classes can be identified with the spaces of functions that
are analytic in the unit disk $\D$: in particular, $H^\infty$ is
the space of all functions bounded and
analytic in $\D$; \medskip

\item[$H^2_E$] vector-valued Hardy class $H^2$ with values in
$E$; \medskip

\item[$H^2_n$] vector-valued Hardy class $H^2$ with values in $\C^n$; \medskip

\item[$L^\infty_{\!E_*\shto E}$] class of bounded functions on the unit
circle $\T$ whose values are bounded operators from $E_*$ to $E$; \medskip

\item[$H^\infty_{\!E_*\shto E}$] operator Hardy class  of bounded analytic
 functions whose values are bounded  operators from $E_*$ to $E$:
$$
\|F\|_\infty := \sup_{z\in \D}
\nm F(z)\nm=\underset{\xi\in
\T}{\operatorname{esssup}}\nm F(\xi)\nm; \medskip
$$\


\item[$T_\Phi$] Toeplitz operator with symbol $\Phi$;  and \medskip

\item[$\tr A$] trace of the operator A.

\medskip

\end{entry}

Throughout the paper all Hilbert spaces are assumed to be
separable. We always assume that in any Hilbert space, an
orthonormal basis is fixed so that any operator $A:E\to E_*$ can
be identified with its matrix. Thus, besides the usual involution
$A\mapsto A^*$ ($A^*$ is the Hilbert-adjoint of $A$), we have two more:
$A\mapsto A^T$ (transpose of the matrix) and $A\mapsto \overline
A$ (complex conjugation of the matrix), so $A^* =(\overline A)^T
=\overline{A^T}$. Although everything in the paper can be
presented in an invariant, ``coordinate-free'' form, the use of
transposition and complex conjugation makes the notation easier
and more transparent.

\setcounter{section}{-1}

\section{Introduction and result}
The main objects of this paper are operators with complete
analytic families of eigenvectors, the backward shift being one of
the simplest examples of such operators. 
Classification of such
operators up to unitary equivalence was completely done by M. J.
Cowen and R. G. Douglas in [4]. They had shown, in particular,
that if the eigenvector bundles of such operators are equivalent
as Hermitian holomorphic vector bundles, then they are unitarily
equivalent. They had also introduced numerous local criteria of the
equivalence of the eigenvector bundles (and so the unitary
equivalence of the corresponding operators).

We are interested in the problem of classification of such
operators up to similarity. Let us recall that operators $T_1$ and
$T_2$ are similar if there exists a (bounded) invertible operator
$A$ such that $T_1 = AT_2A^{-1}$. It was shown already in
\cite{CowenDouglas} that this problem is significantly more
complicated than unitary classification; in particular, it was
shown in \cite{CowenDouglas} that similarity of operators cannot
be expressed as a local condition on their eigenvector bundles.
So, we restrict ourselves to a particular case of the general
problem. Namely, we are interested in the case when an operator is
similar to the backward shift $S^*$ in the  Hardy
space $H^2$ (scalar or vector valued). 

Let us recall that the backward shift $S^*$ in the  Hardy space $H^2_E$ 
is the adjoint of the \emph{forward shift} $S$, $Sf(z) =zf(z)$, $f\in H^2$, and can be expressed as
$S^*f(z) =(f(z) -f(0))/z$, $f\in H^2$. The same formulas can be used to define $S^*$ on the vector Hardy space $H^2_E$. Sometimes, to emphasize that we are considering $S^*$ in the vector Hardy class $H^2_E$ we will use the notation $S^*_E$ (or $S^*_n$ if $\dim E=n$). 

Eigenvectors of $S^*$ are well known. Namely, the point spectrum (the set of eigenvalues) of $S^*$ is the open unit disc $\D$, and $S^*f =\la f$, $f\in H^2_E$, $|\la|<1$, if and only if $f$ can be represented as
$$
k_{\overline\la} e, \qquad e\in E;
$$
here $k_\la$ denotes  the reproducing kernel
\footnote{The function $k_\la$ is called the reproducing kernel because  
$(f, k_\la)= f(\la)$ for $f\in H^2$.  This property explains why  the notation $k_\la$ and not $k_{\overline \la}$ is used for the function $1/(1-\overline\la z)$. }
 of the (scalar) Hardy space $H^2$, $k_\la(z) := 1/(1-\overline\la z)$. Probably the easiest way to see that is to represent $k_{\overline\la}e $ as the geometric series, $k_{\overline\la}e =\sum_0^\infty \la^k z^k e$. 

We will also assume
that
the operator $T$ is contractive, i.e.,~that $\|T\|\le 1$.
As one can see from our result, one cannot expect a simple solution for the general case. 

One can say that for this case the problem was solved by
B.~Sz.-Nagy and C.~Foia\c{s} 
[9, Chap 9.1],\cite{NagyFoias2},\cite{NagyFoias3}
who proved that an operator $A$, $\|A\|\le1$, in a separable Hilbert space is similar
to an isometry if and only if its characteristic function is left
invertible in (operator-valued) $H^\infty$. Under our 
assumptions about $T$ the operator $T^*$ is completely non-unitary, 
so the similarity of $T^*$ to an isometry is equivalent to the similarity 
to the   forward shift
$S$, $Sf(z) =zf$, $f\in H^2_E$, in some (generally vector-valued)
$H^2$ space. Taking the adjoint operators, we obtain that $T$ is similar to the backward shift $S^*$ in 
some vector-valyed space $H^2$ if and only if the characteristic function of the operator $T^*$ is left invertible in $H^\infty$. 

However, we are interested in the
description only in terms of the geometry of the eigenvector
bundle, and the result of B.~Sz.-Nagy and C.~Foia\c{s} does not give
such
a description.

We assume the following about our linear operator $T$ on
a Hilbert space $H$:
\begin{enumerate}

\item
$T$ is contractive, i.e., $\|T\|\leq 1$;

\item
$\dim \ker(T-\lambda I)$ is constant for all $\lambda\in
\D$;

\item
$\spn \{\ker (T-\lambda I):\lambda\in \D\}=H$ ; and

\item the subspaces 
$\cE(\la)=\ker(T-\lambda I)$ form a  depend analytically on the spectral parameter
$\la \in \D$.\
\end{enumerate}

\medskip
Assumption \cond4 means that for each $\omega \in\D$ there exists a neighborhood $U_{\omega}$ of $ \omega$ and a left invertible in $L^\infty$ 
analytic operator-valued function $F_{\omega} $ defined   on $U_\omega$, $F_\omega(\la):E_*\to E$, such
that $\ran F_{\omega}(\lambda)=\cE(\lambda)$. It is easy to see that $\dim E_*$ must be the same for all $\omega$, so $\dim \cE(\la ) = \dim E_*$ for all $\la\in \D$, so condition \cond2 is redundant. 

If $\dim E_*<\infty$, then the columns of $F(\la)$ form a basis in $\cE(\la )$, so the disjoint union $\coprod_{\la\in\D} \cE(\la) =\{(\la, v_\la) : \la\in \D, v_\la \in \cE(\la)\}$ is a holomorphic vector bundle over $\D$ (subbundle of the trivial bundle $\D\times H$) with the natural projection $\pi$, $\pi (\la, v_\la) =\la$.  

In the case $\dim E_*=\infty$, the above statement can be used as a definition of a holomorphic vector bundle of infinite rank.

We will follow the usual agreements and write $v_\la$ instead of $(\la, v_\la)$, which simplifies the notation. Note, that the subspaces  $\cE(\la)$ inherit the metric from the Hilbert space $H$, so our bundle $\coprod_{\la\in\D} \cE(\la) $ is a Hermitian holomorphic vector bundle.


One can state some assumptions about the operator $T$ that guarantee that condition \cond 4 holds. 
For example, it is proven in \cite{CowenDouglas} that for a bounded linear operator $T:H\to H$ such that for all 
$\lambda \in \D$ the operator $T-\lambda I$ is Fredholdm, $\ran (T-\lambda I) = H$ and $\dim\ker (T-\la I) \equiv \text{const}$, condition \cond4 holds.

In order to state the result of the paper, we define on the unit
disk $\D$ a projection-valued function $\Pi:\D\rightarrow
B(H)$ that assigns to each $\lambda \in \D$, the orthogonal
projection onto $\ker (T-\lambda I)$;
$$
\Pi(\lambda):= P_{\ker(T-\lambda I)}.
$$
This function is clearly $C^\infty$ and even real analytic in the operator norm topology, but in what follows we will only need the fact that   it is a
$\cC^2$ function, i.e., a function twice continuously
differentiable (in the operator norm topology).

Let us also recall that if $\cE$ and $\wt\cE$ are two holomorphic vector  bundles over the same set $\Omega$, then a  map   $\Psi: \cE\to \wt\cE$ is called a \emph{bundle map} if it is holomorphic, and for each $\la \in \Omega$ the restriction of $\Psi$ onto the  fiber $\cE(\la) := \pi^{-1}(\la)$ is a linear transformation from $\cE(\la)$ to $\wt\cE(\la) = \wt\pi^{-1}(\la)$. 

Now we are ready to state our main result:
\begin{thm} \label{t0.1} Let $T$ be a linear operator on a Hilbert space $H$
under assumptions (1) through (4) such that
$\dim\ker(T-\lambda I)=n<\infty$ for every $\lambda\in\D$. Let $\Pi(\lambda)$ be the orthogonal projection onto
$\ker(T-\lambda I)$. Then the following statements are equivalent:

\begin{enumerate}
\item
T is similar to the backward shift operator $S^*_n$ on $H^2_n$
via an invertible operator $A:H^2_n \rightarrow H$;

\item
The eigenvector bundles of $T$ and $S^*_n$ are ``uniformly equivalent''. i.e., there exists a holomorphic bundle map bijection $\Psi$ from the eigenvector bundle of
$S^*_n$ to that of $T$  such that for some constant $c>0$, 
$$
\frac{1}{c} \| v_{\lambda}  \|_{H^2_n} \leq \| \Psi ( v_{\lambda} )\|
_{H} \leq c \| v_{\lambda} \| _{H^2_n} 
$$
for all $v_{\lambda} \in \ker(S^*_nI-\lambda)$ and for all $\la\in \D$; 

\item
There exists a bounded subharmonic function $\vf$ such that
$$
\Delta \vf (z) \ge \left \bracevert \frac{\partial\Pi(z)}{\partial
z}\right \bracevert ^2_{\mathfrak{S}_2}-\frac{n}{(1-|z|^2)^2}
\quad \text{ for all } z\in \D.
$$
\item 
The measure 
$$
\left(\left \bracevert
\frac{\partial\Pi(z)}{\partial z} \right \bracevert
^2_{\mathfrak{S}_2}-\frac{n}{(1-|z|^2)^2}\right)(1-|z|)dxdy
$$ 
is  Carleson and the estimate 
$$
\left(\left\bracevert
\frac{\partial\Pi(z)}{\partial z} \right \bracevert
^2_{\mathfrak{S}_2}-\frac{n}{(1-|z|^2)^2}\right)^{\frac{1}{2}} \le
\frac{C}{1-|z|}
$$
holds.
\end{enumerate}

\end{thm}

\begin{rem}
\label{rem0.2}We see in Section 1 that $\left \bracevert
\frac{\partial\Pi(z)}{\partial
z}\right \bracevert ^2_{\mathfrak{S}_2}-\frac{n}{(1-|z|^2)^2}\ge 0$.
\end{rem}

\begin{rem}
\label{rem0.3} Treating $\ker(T-\lambda I)$ as a subbundle of the trivial  bundle $H\times\D$, 
one can see that $-\frac{\partial\Pi(\lambda)}{\partial
\lambda}$ is its second fundamental form, so the mean curvature of
the eigenvector bundle $\ker(T-\lambda)$ is $-\left
\bracevert \frac{\partial\Pi(\lambda)}{\partial \lambda}\right
\bracevert ^2_{\mathfrak{S}_2}$. On the other hand,
$-(1-|\lambda|^2)^{-2}=\Delta \ln
({1-|\lambda|^2})$ is the curvature of the eigenvector
bundle of $S^*$, so $-n(1-|\lambda|^2)^{-2}$ is the mean curvature of the eigenvector bundle of the backward shift $S^*_n$ of multiplicity $n$. 

Thus, statements \cond 3 and \cond 4 are about the mean curvatures of the eigenvector bundles of $T$ and $S^*$. 
\end{rem}

\begin{rem}
\label{rem0.4}
Statement (3) of the theorem simply  means that the Green potential
$$
\mathcal G(\la) := \frac2\pi\iint_\D \ln \left|
\frac{z-\la}{1-\overline \la z}\right| \left (\left \bracevert
\frac{\partial\Pi(z)}{\partial z} \right \bracevert
^2_{\mathfrak{S}_2}-\frac{n}{(1-|z|^2)^2} \right)dxdy
$$
is uniformly bounded inside the unit disk $\D$. Integrating
separately over a small neighborhood of $\lambda$ and the rest of
$\D$, one can  easily see that \cond 4 $\implies$ \cond 3. 
\end{rem}

\section{Preliminaries} \subsection{Inner-outer factorization and invariant subspaces} Let us recall that an operator-valued function $F\in
H^\infty_{E_*\shto E}$ is called \emph{inner} if $F(z)$ is an
isometry a.e. ~on~$\T$, and  \emph{outer} if $FH_{E_*}^2$ is dense
in $H_{E}^2$. Every $F \in H^\infty_{E_*\shto E}$ can be
represented as $F=F_iF_o$ for an inner function $F_i \in
H^\infty_{E_{\#}\shto E}$, an outer function $F_o \in
H^\infty_{E_*\shto E_{\#}}$, and an auxiliary Hilbert space
$E_{\#}$. 

Let $S=S_E$ be the (forward) shift operator on $H^2_E$, $Sf =zf$, $f\in H^2_E$. The famous Beurling--Lax Theorem says that any non-zero $S$-invariant  subspace $\mathcal{E}\subset H^2_E$, $S\mathcal E\subset\mathcal E$,  can be represented as $\Theta H^2_{E_*}$, where $E_*$ is an auxiliary Hilbert space and $\Theta \in H^\infty_{E_*\shto E}$ is an inner function. The inner function $\Theta$ is unique up to a constant unitary factor on the right. 

The backward shift $S^*$, $S^*f=(f(z)-f(0))/z$, $f\in H^2_{E}$, is the adjoint of $S$, so any non-trivial invariant subspace $K\subset H^2_E$ of $S^*$ admits the representation $K=K_\Theta := H^2_E\ominus\Theta H^2_E$
 with some inner function $\Theta \in H^\infty_{E_*\shto E}$.

 \subsection{Tensor structure of the eigenvector bundle of $T$}
 \label{s1.1}   
The following 
theorem [8, Chap 0.2] plays a critical role in what follows.

\begin{thm}[Model Theorem] Every contraction $T$ on $H$ with the property that
$\lim_n \|A^n h\|=0$ for every $h\in H$ is unitarily equivalent to
$S^*_E\big|  K$ for some Hilbert space $E$ and an $S^*_E$-invariant
subspace $K$ of $H^2_E$.
\end{thm} 

For our operator $T$ we trivially have $\lim_n \|T^n h\|=0$ for linear combinations of eigenvectors, which are dense in $H$ by assumption \cond3.  
 Since $T$ is a contraction,   $\|T^n\|\leq 1 $, so the standard 
$\e/3$ argument shows that the conditions of
Theorem 1.1. are satisfied.  So, without loss of generality, we can assume that $T$ is the  restriction of the backward shift $S^*$ in the vector Hardy space $H^2_E$ (where $E$ is an auxiliary Hilbert space) onto its invariant subspace $K\subset H^2_E$.  If $K=H^2_E$ the operator $T$ is  the backward shift, so we only need to consider the case when  $K$ is a proper subspace of $H^2_E$. In this case  $K$ can be represented as $K=K_\Theta$, where $\Theta\in H^\infty_{E_*\shto E} $ is an inner function.

Clearly an eigenvector of $T$ is an eigenvector of $S^*$, and the eigenvectors of $S^*$ are well known. As it was shown before in the introduction the eigenspace $\ker (S^*-\la I)$ of the backward shift $S^*$ in the scalar Hardy space $H^2$ 
spanned by the reproducing kernel $k_{\bar \la}$, where recall $k_\la(z) = 1/(1-\bar\la z)$. So in the case of the backward shift $S^*_E$ in $H^2_E$, 
$$
\ker(S^*_E - \la I)  =\{k_{\overline \la}(z) e: e\in E\}\,. 
$$

So, the eigenspaces of $T=S^*\big| K$ are given by 
$$
\ker(T-\la I) = k_{\overline\la} E(\la) =\{ k_{\overline\la} e: e\in E(\la)\},
$$
where the $E(\la)$ are some subspaces of the space $E$. 
The assumption \cond 3 that $\ker(T-\la I)$ is a holomorphic vector bundle implies that the subspaces $E(\la)$ depend analytically on the spectral parameter $\la$, i.e., that the family of subspaces $E(\la)$ is a holomorphic vector bundle as well.  

The vector valued Hardy space $H^2_E$ is a natural realization of the tensor product $H^2\otimes E$, so we can write 
$$
\ker(T-\la I) = k_{\overline\la} \otimes E(\la). 
$$

\begin{rem}
While it is not essential for the proof of the main result (Theorem \ref{t0.1}), it is easy to see that $E(\la)= \ker \Theta(\overline\la)^*$, where $\Theta\in H^\infty_{E_*\shto E}$ is the inner function such that $K=K_\Theta$. Indeed, an eigenvector $k_{\overline\la} e$ belongs to $K_\Theta$ if and only if $k_{\overline\la}e\perp \Theta H^2_{E_*}$.  Using the reproducing kernel property of $k_\la$ we get that for $h\in H^2_{E_*}$,
\begin{equation}
\La \Theta h, k_{\overline\lambda} e \Ra_{H^2_E} = 
\La \Theta(\overline\lambda)h(\overline\lambda),e \Ra_E 
=\La h(\overline\lambda),\Theta(\overline\lambda)^* e \Ra_E.
\end{equation}
Since $\{h(\overline\la):h\in H^2_{E_*}\}=E$,
we conclude that $k_{\bar\la} e \perp\Theta H^2$ iff $\Theta(\bar\la)^* e=0$.
Therefore, $E(\la)= \ker\Theta(\bar\la)^*$. 
\end{rem}

\begin{rem}The inner function $\Theta \in H^{\infty}_{E_* \rightarrow E}$
appearing above is the characteristic
function of the operator $T^*$, and therefore the spaces $E_*$ and $E$ can be identified with
$\clos(I-TT^*)^{\frac{1}{2}}H$ and $\clos (I-T^*T)^{\frac{1}{2}}H$, respectively [9, Chap 6.2].
\end{rem}

\subsection{Curvature of the eigenvector bundle of $T$} Let us compute the norm $\left\bracevert \p\Pi/\p z\right\bracevert_{\mathfrak{S}_2}^2$, where $\Pi(z)$ is the orthogonal projection onto $\ker (T-zI)$.  As we mentioned before, this expression is the mean curvature of the eigenvector bundle of $T$. 

Using the tensor structure  $\ker (T-\la I) = k_{\bar\la} \otimes E(\la)$ one can represent $\Pi(\la)$ as 
\begin{equation}
\label{Pi-tensor}
\Pi(\la)=\Pi_1(\la)\otimes\Pi_2(\la), 
\end{equation}
where $\Pi_1(\la)$ it the orthogonal projection in the (scalar) space $H^2$ onto $\spn\{k_{\bar\la}\}$, and 
$\Pi_2(\la)$ is the orthogonal projection in $E$ onto $E(\la)$. 

\begin{lm}
\label{l-curv-T}
In the above notation, if $\rank \Pi(\la) (=\rank \Pi_2(\la))= n<\infty$, then 
\begin{align*}
{\left \bracevert \frac{\partial\Pi(\lambda)}{\partial\lambda}\right \bracevert}^2_{\mathfrak{S}_2} & = 
{\left \bracevert \frac{\partial\Pi_1(\lambda)}{\partial\lambda}\right \bracevert}^2_{\mathfrak{S}_2} +
{\left \bracevert \frac{\partial\Pi_2(\lambda)}{\partial\lambda}\right \bracevert}^2_{\mathfrak{S}_2}
\\
& = 
\frac{n}{(1-|z|^2)^2} \, + 
{\left \bracevert \frac{\partial\Pi_2(\lambda)}{\partial\lambda}\right \bracevert}^2_{\mathfrak{S}_2}.
\end{align*}
\end{lm}
To prove this lemma we will need a couple of well known and simple facts. 

\begin{lm}
\label{l-PdP} Let $E(\la)$, $\la\in \D$, be an analytic family of subspaces (holomorphic vector bundle), and let $\Pi(\la) $ be the orthogonal projection onto $E(\la)$. Then 
$$
(I-\Pi(z))\frac{\p\Pi(z)}{\p z} \Pi(z) = \frac{\p\Pi(z)}{\p z}. 
$$
\end{lm}

This lemma is a well known fact in complex differential geometry, but for the sake of completeness we present the proof. 

\begin{proof}[Proof of Lemma \ref{l-PdP}]
The fact that the family of subspaces $E(\la)$ is a holomorphic vector bundle means that locally  the subspaces $E(\la)$ can be represented as $\ran F(\la)$ where $F$ is a left invertible analytic operator-valued function. Given such a representation one can write the formula for $\Pi$, $\Pi=F(F^*F)^{-1}F^*$. Direct computation shows that 
$$
\frac{\partial\Pi(z)}{\partial z}=(I-\Pi(z))F'(z)(F(z)^*F(z))^{-1}F(z)^*
$$
and the conclusion of the lemma follows immediately. 
\end{proof}
\begin{lm}
\label{l-AxB-HS}
For operators $A$ and $B$ in the Hilbert-Schmidt class
$\mathfrak{S}_2$,
$$
{\| A \otimes B}\|^2_{\mathfrak{S}_2}=\|A
\|^2_{\mathfrak{S}_2}\| B\|^2_{\mathfrak{S}_2}.
$$
\end{lm}

This lemma is well known and the proof is an easy exercise, so we omit it. 

\begin{proof}[Proof of Lemma \ref{l-curv-T}]

Using the product rule, we get from \eqref{Pi-tensor}
\begin{equation}
\frac{\partial\Pi(\la)}{\partial \la}
=\frac{\partial\Pi_1(\la )}{\partial \la }
\otimes\Pi_2(\la )+\Pi_1(\la )\otimes\frac{\partial\Pi_2(\la )}{\partial
\la } =: X + Y.
\end{equation}
The identity $\Pi_2(\la) \frac{\p \Pi_2(\la)}{\p \la} = 0$ (see Lemma \ref{l-PdP}) implies that $X^*Y=0$. Therefore
\begin{align*}
\left\bracevert X+Y\right\bracevert_{\mathfrak{S}_2}^2 = \tr X^*X + \tr Y^*Y + 2\re\tr (X^*Y) &
= \tr X^*X + \tr Y^*Y
\\
& = \left\bracevert X\right\bracevert_{\mathfrak{S}_2}^2 + \left\bracevert Y\right\bracevert_{\mathfrak{S}_2}^2
\end{align*}
Applying Lemma \ref{l-AxB-HS} to each term and recalling that for an orthogonal projection $P$ we have $\left\bracevert P\right\bracevert_{\mathfrak{S}_2}^2 = \rank P$, we get that 
$$
\left\bracevert \frac{\partial\Pi(\la)}{\partial \la} \right\bracevert_{\mathfrak{S}_2}^2 = n 
\left\bracevert \frac{\partial\Pi_1(\la)}{\partial \la} \right\bracevert_{\mathfrak{S}_2}^2 + 
\left\bracevert \frac{\partial\Pi_2(\la)}{\partial \la} \right\bracevert_{\mathfrak{S}_2}^2\,. 
$$
The lemma is proved modulo computation of $\left\bracevert \frac{\partial\Pi_1(\la)}{\partial \la} \right\bracevert_{\mathfrak{S}_2}^2$, which is done in the next lemma. 
 \end{proof}

\begin{lm}
\label{curv-shift}
Let $\Pi_1(\la)$ be the orthogonal projection onto $\spn\{k_{\bar\la}\}$ in $H^2$ (scalar valued). Then 
$$
\left \bracevert\frac{\partial\Pi_1(\lambda)}{\partial \lambda}
\right \bracevert^2_{\mathfrak{S}_2}
={(1-|\lambda|^2)^{-2}}\qquad\text{for all }\lambda \in \D.
$$
\end{lm}

\begin{proof}[Proof of Lemma \ref{curv-shift}]
The proof can a simple exercise in complex differential geometry, using the fact that the quantity in question is (up to the sign) the curvature of the eigenvector bundle of $S^*$; see Remark \ref{rem0.3}.  

However, for the convenience of the reader we present a direct computation (one of the  many possible). 

First, recall that $k_\la$ is the \emph{reproducing kernel} of $H^2$, i.e., $\La f, k_\la\Ra= f(\la)$ for all $f\in H^2$. Using the reproducing kernel property of $k_\la$ we conclude that $\|k_\la\|_2^2 = \La k_\la, k_\la\Ra = (1-|\la|^2)^{-1}$. Therefore for $f\in H^2$,
$$
\Pi_2(\la) f =\|k_{\bar\la}\|_2^{-2} \La f, k_{\bar\la} \Ra k_{\bar\la} = (1-|\la|^2) f(\bar\la) k_{\bar\la}.
$$
Taking $\frac{\p}{\p\la}$ and using the fact that $\frac{\p f(\bar\la)}{\p\la}=0$ we get 
$$
\frac{\p\Pi_2(\la)}{\p\la} f = f(\bar\la) \left( -\bar\la k_{\bar\la} + (1-|\la|^2) \wt k_{\bar\la} \right), 
$$
where
$$
\wt k_{\bar\la}(z) =\frac{\p}{\p\la} k_{\bar\la}(z) = \frac{z}{(1-\la z)^2}.  
$$ 
Note, that for $f\in H^2$,
\begin{equation}
\label{repr-k1}
\La f, \widetilde k_{\la}\Ra = f'(\la). 
\end{equation}

Using this identity one can get that 
$$
\|\wt k_{\la} \|_2^2 = \frac{1+|\lambda|^2}{(1-|\lambda|^2)^3} 
= \|\wt k_{\bar\la} \|_2^2.
$$
The reproducing property for $k_\la$ implies that 
$$
\La\wt k_{\bar\la} , k_{\bar\la}\Ra = \frac{\bar\la}{(1-|\la|^2)^2}.
$$

Combining all together we can conclude that 
$$
\|-\bar\la k_{\bar\la} + (1-|\la|^2) \wt k_{\bar\la}\|_2^2 = (1-|\la|^2)^{-1}.
$$

Since $f(\bar\la) = \La f, k_{\bar\la}\Ra$, and as we discussed above $\|k_{\bar\la}\|_2^2 = \|k_{\la}\|_2^2 = (1-|\la|^2)^{-1}$, 
$$
\left \bracevert\frac{\partial\Pi_1(\lambda)}{\partial \lambda}
\right \bracevert^2 = (1-|\lambda|^2)^{-2}. 
$$
Since (see Lemma \ref{l-PdP}) $\frac{\p\Pi(\la)}{\p\la}$
is a rank one operator, its operator and Hilbert--Schmidt norms coincide. 
\end{proof}

\section{From uniform equivalence of bundles to curvature condition}

In this section we are going to prove the implication \cond2$\implies$\cond4. 

Note that according to Lemma \ref{l-curv-T}
$$
{\left \bracevert \frac{\partial\Pi_2(\lambda)}{\partial\lambda}\right \bracevert}^2_{\mathfrak{S}_2} =
{\left \bracevert \frac{\partial\Pi(\lambda)}{\partial\lambda}\right \bracevert}^2_{\mathfrak{S}_2} -
\frac{n}{(1-|z|^2)^2} ,\, 
$$
so ${\left \bracevert \frac{\partial\Pi(\lambda)}{\partial\lambda}\right \bracevert}^2_{\mathfrak{S}_2} -
\frac{n}{(1-|z|^2)^2}$ in statements \cond3 and \cond4 of Theorem \ref{t0.1} can be replaced by the curvature  ${\left \bracevert \frac{\partial\Pi_2(\lambda)}{\partial\lambda}\right \bracevert}^2_{\mathfrak{S}_2}$. 

Let $\Psi$ be the uniformly equivalent bundle map bijection, as in condition \cond2. A bundle map means that $\Psi$ is an analytic function of $\la$, maps the  fiber $\ker(S^*_n-\la I)$ to the fiber $\ker(T-\la I)$ and is linear in each fiber $\ker(S^*_n-\la I)$. 

It is easy to see from the descriptions of $\ker(S^*_n-\la I)$ and $\ker(T-\la I)$ that any such bundle map bijection is represented by  
$$ 
\Psi(k_{\bar{\lambda}}e)=k_{\bar{\lambda}} \cdot F(\lambda)e, \qquad \forall e\in \C^n, 
$$
where    $F \in
H^{\infty}_{\C^n \rightarrow E}$ is an operator-valued function such that 
$$
\ran F(\lambda)= E(\la) \ (=\ker \Theta(\bar{\lambda})^*).
$$ 
The ``uniform equivalence'' property of $\Psi$ means that 
\begin{equation}
\label{un_eq}
c^{-1} I \le F^*F\le CI, \qquad \forall z\in \D. 
\end{equation}
Hence the orthogonal projection 
$\Pi_2(\la)$ from $E$ onto $E(\la)$ can be written down  as
$$
\Pi_2=F(F^*F)^{-1}F^*.
$$
Differentiating we get  $\frac{\partial\Pi_2(z)}{\partial
z}=(I-\Pi_2(z))F'(z)(F(z)^*F(z))^{-1}F(z)^*$, and taking into account \eqref{un_eq} we have
\begin{equation}
\label{2.2}
\left \bracevert \frac{\partial\Pi_2(z)}{\partial z} \right \bracevert \leq C\nm F'(z) \nm.
\end{equation}
The function $F $ takes values in
the Hilbert-Schmidt class $\mathfrak{S}_2$ which is a  Hilbert space, and for bounded analytic functions with values in a Hilbert space the estimate
$$
\nm F'(z) \nm \leq {C}/{(1-|z|)}
$$
holds, and the measure 
$$
 \nm F'(z) \nm ^2(1-|z|)dxdy 
$$
is Carleson. Combining these facts with \eqref{2.2} we conclude that the curvature condition \cond 4 holds.

\section{Curvature condition implies  similarity}
As we already mentioned in the Introduction, see Remark \ref{rem0.4} there, it is easy to show that condition \cond 4 implies condition \cond3. 

As it was already discussed in the beginning of the previous section, the expression ${\left \bracevert \frac{\partial\Pi(\lambda)}{\partial\lambda}\right \bracevert}^2_{\mathfrak{S}_2} -
\frac{n}{(1-|z|^2)^2}$ in statements \cond3 and \cond4 of Theorem \ref{t0.1} can be replaced by  ${\left \bracevert \frac{\partial\Pi_2(\lambda)}{\partial\lambda}\right \bracevert}^2_{\mathfrak{S}_2}$. So, the implication \cond3$\implies$\cond1 follows from the theorem below, which holds even in the case $\dim E(\la) = \infty$.

\begin{thm}
\label{t3.1}
Let $E(\la)$, $\la\in \D$, be an analytic family of subspaces of a Hilbert space $E$, and let $\Pi(\la)$ be the orthogonal projection onto $E(\la)$. Let
$$
T = S^*_E \big| K, 
$$
where $K$ is the $S^*$-invariant subspace of $H^2_E$, $K:= \spn\{ k_{\bar\la} e: \la\in \D, e\in E(\la)\}$. Suppose that there exists
a bounded, subharmonic function $\vf$ such that
$$
\Delta \vf (z) \ge \left \bracevert
\frac{\partial\Pi_2(z)}{\partial z}\right \bracevert^2 \quad
\forall z\in \D.
$$
Then $T$ is similar to the backward shift $S^*_{E_*}$, where $E_*$ is an auxiliary Hilbert space and $\dim E_*=\dim E(\la)$. 
\end{thm}

\subsection{Toeplitz operators}
To prove Theorem \ref{t3.1} we will need to recall some simple facts about Toeplitz operators. 
Let us recall that given an operator-valued function 
$F\in L^\infty{E\shto E_*}$, the Toeplitz operator $T_F:H^2_E\to H^2_{E_*}$ with symbol $F$ is defined by the formula
$$
T_F f = P_+ (Ff),\qquad f\in H^2_E
$$
where $P_+$ is the orthogonal projection in $L^2$ onto $H^2$. 

If the symbol $F$ is analytic ($F\in H^\infty_{E\shto E_*})$, then the Toeplits operator $T_F$ is  simply  the multiplication operator (more precisely, its restriction onto $H^2$). 

It is easy to see that if $F, G\in H^\infty$, then 
\begin{equation}
\label{3.1}
T_F T_G = T_{FG}.
\end{equation}
We will also need the following well known and easy to prove fact that 
in $F\in H^\infty_{E\shto E_*}$, then 
\begin{equation}
\label{3.2}
T_{F^*} k_\la e = k_\la F^*(\la)e, \qquad e\in E_*
\end{equation}
and
\begin{equation}
\label{3.3}
T_{F^*} S^*_{E_*} = S^*_E T_{F^*}. 	
\end{equation}

\subsection{Proof of Theorem \ref{t3.1}}
 We want to prove the
existence of an invertible operator $A:H^2_n \to K$ satisfying the
intertwining relation $AS^*_n=TA$ (recall that $T=S^*_E  \big| K)$.

We will need the following theorem by  S. Treil and B.
D. Wick [17]. 

\begin{thm}
\label{t_Tr-W}
Let $\Pi:\D \to B(E)$ be a $\cC^2$ function whose values are
orthogonal projections in $E$, satisfying the identity $\Pi(z)
\frac{\p\Pi(z)}{\p z}=0$. Assume that for some bounded subharmonic
function $\vf$, we have
$$
\Delta \vf (z) \ge \left \bracevert \frac{\p\Pi(z)}{\p z}\right \bracevert^2\qquad
\text{for all }z\in \D.
$$

Then there exists a bounded analytic projection onto $\ran
\Pi(z)$, i.e., a function $\cP\in H^\infty_{E \shto E}$ such that
$\cP(z)$ is a projection onto $\ran\Pi(z)$ for all $z\in \D$.
\end{thm}

By Lemma \ref{l-PdP}, the function $\Pi$ from Theorem \ref{t3.1} satisfies the identity $\Pi(z)
\frac{\p\Pi(z)}{\p z}=0$, so applying Theorem \ref{t_Tr-W}  to it we get a bounded analytic projection $\cP(z)$ onto $\ran \Pi(z)$. Consider the inner-outer factorization $\cP = \cP\ti i \cP\ti o$ of $\cP$, where $\cP\ti i\in H^\infty_{E_*\shto E}$ is the inner part and $\cP\ti o\in H^\infty_{E\shto E_*}$ is the outer part of $\cP$. 
Define a function $\cP\ti i^\sharp$ by $\cP\ti i^\sharp (z) := \cP\ti i(\bar z)$, and consider the Toeplitz operator $T_{\cP\ti i^\sharp}$. Since $(\cP\ti i^\sharp)^* \in H^\infty_{E\shto E_*}$, \eqref{3.3} implies 
\begin{equation}
\label{3.4}
T_{\cP\ti i^\sharp} S^*_{E_*} = S^*_{E} T_{\cP\ti i^\sharp}.
\end{equation}
If we show that the operator $T_{\cP\ti i^\sharp}$ is left invertible and that $\ran T_{\cP\ti i^\sharp} = K$, we are done: the operator $A$ we want to find is simply the Toeplitz operator $T_{\cP\ti i^\sharp}$ treated as an operator $H^2_{E_*}\to K$. The left invertibility of $T_{\cP\ti i^\sharp}$ together with $\ran T_{\cP\ti i^\sharp} = K$ means that $A$ is invertible, and the intertwining $AS^*_{E_*} =TA$ follows from \eqref{3.4}. 

The left invertibility of $T_{\cP\ti i^\sharp}$ is a corollary of the following lemma.

\begin{lm}
\label{l3.3}
The outer part $\cP\ti o$ is a left inverse of $\cP\ti i$, i.e., $\cP\ti o \cP\ti i \equiv I$ for all $z\in \D$. 
\end{lm}

This lemma immediately implies that $T_{\cP\ti o^\sharp}$, where $\cP\ti o^\sharp(z) :=\cP\ti 0(\bar z) $, is a left inverse of $T_{\cP\ti i^\sharp}$. Indeed
$$
(T_{\cP\ti i^\sharp})^*(T_{\cP\ti o^\sharp})^* = T_{(\cP\ti i^\sharp)^*}T_{(\cP\ti o^\sharp)^*} = T_{(\cP\ti i^\sharp)^*(\cP\ti o^\sharp)^*} = I. 
$$
The last equality holds because $\cP\ti o \cP\ti i \equiv I$ and $T_I = I$; the previous one follows from \eqref{3.1} because $(\cP\ti i^\sharp)^*, (\cP\ti o^\sharp)^*\in H^\infty$. 

\begin{proof}[Proof of Lemma \ref{l3.3}]
It follows from \eqref{3.1} that 
$$
T_{\cP\ti i}  T_{\cP\ti o}  = T_\cP = T_{\cP^2} = T_{\cP\ti i \cP\ti o \cP\ti i \cP\ti o} = T_{\cP\ti i} T_{\cP\ti o \cP\ti i} T_{\cP\ti o} .
$$
The operator $T_{\cP\ti o}$ has dense range because $\cP\ti o$ is outer, and $\ker T_{\cP\ti i} =\{0\}$ because $\cP\ti i$ is inner (in fact, $T_{\cP\ti i} =\{0\}$ is an isometry).
Therefore $T_{\cP\ti o \cP\ti i} = I$, so $\cP\ti o \cP\ti i \equiv I$ for all $z\in \D$. 
\end{proof}

To complete the proof of Theorem \ref{t3.1} it remains to show that $\ran T_{\cP\ti i^\sharp} = K$. First of all, let us notice that 
$$
\ran \cP\ti i(\la) = E(\la) \qquad \forall \la \in\D. 
$$
Indeed, the inclusion $E(\la) = \ran \cP(\la) \subset \ran\cP\ti i(\la)$ is trivial because of the factorization $\cP= \cP\ti i \cP\ti o$. Since $\cP\ti o$ is outer, the set $\ran \cP\ti o(\la)$ is dense in $E_*$ for all $\la \in \D$. 
But $\cP\ti i(\la) \ran\cP\ti o(\la) = E(\la)$, so  $\ran \cP\ti i(\la) \subset E(\la)$. 
It follows from \eqref{3.2} that 
\begin{equation}
\label{3.5}
	T_{\cP\ti i^\sharp} k_{\bar\la} e = k_{\bar\la} \cP\ti i^\sharp(\bar\la)e = k_{\bar\la} \cP\ti i(\la)e. 
\end{equation}
To see that we got all the complex conjugates correctly, one can fix bases in $E$ and $E_*$, consider  the matrix representation of the operators, and notice that $\cP^\sharp(z) = (\cP^T(z))^*$. 
Then a direct  application of \eqref{3.2} implies \eqref{3.5}.

The equality $\ran\cP\ti i(\la) =E(\la)$ together with \eqref{3.5} imply that 
$$
T_{\cP\ti i^\sharp}k_{\bar\la}\otimes E_* = k_{\bar\la}\otimes E(\la). 
$$
Since $\spn\{k_{\bar\la}\otimes E_*:\la\in\D\} = H^2_{E_*}$, $\spn\{k_{\bar\la}\otimes E(\la):\la\in\D\} = K$, and the operator $T_{\cP\ti i^\sharp}$ is left invertible, we conclude that $\ran T_{\cP\ti i^\sharp} = K$. \hfill \qed

\section{Connection with a result by B. Sz.-Nagy and C. Foia\c{s}}

We already mentioned in the Introduction the following result by B. Sz.-Nagy and C. Foia\c{s} 
 [6, Chap 1.5], [9, Chap 9.1], [10], \cite{NagyFoias3}. 

\begin{thm}
\label{t_N-F}
A contraction $A$  ($\|A\|\le1$) in a Hilbert space is similar to an isometry if and only if its characteristic function is left invertible in $H^\infty$.   
\end{thm}

We are not giving the definition of the characteristic function of a contraction here, because it is quite technical and is not relevant to our paper. The reader only needs to know that if $A=T^*$, where $T$ is the restriction of the backward shift $S^*$ onto an $S^*$-invariant subspace $K=K_\Theta\subset H^2_E$, then the inner function $\Theta$ is the characteristic function of $A$. 

If the operator $A=T^*$ is unitarily equivalent to an isometry $U$, 
the isometry $U$ must be unitarily equivalent to the forward shift $S_{E_*}$ in  $H^2_{E_*}$. 
Indeed, since $\lim_n \|T^n f \| = 0$ for all $f\in K$, the same holds for $U^*$, 
$\lim_n\|(U^*)^n x\|= 0$ for all $x$. 
But it is a well known fact (an easy corollary of  the Kolmogorov--Vold decomposition of isometries) that such isometry $U$ is unitarily equivalent to the forward shift $S$ in $H^2_{E_*}$. 
So, applying Theorem \ref{t_N-F} to our situation we get that $T$ is similar to the backward shift if and only if the inner function $\Theta$ is left invertible in $H^\infty$. 

We would like to investigate what  the relation between this statement and our result is. The following remarkable lemma by N. Nikolski provides that necessary connection.

\begin{lm}
Let $F\in H^\infty_{E_*\shto E}$ satisfy
$$
F(z)^*F(z)\ge \delta^2 I \quad\text{for all } z\in \D.
$$
Then $F$ is left invertible in $H^{\infty}$, i.e.,~there exists a $G\in
H^\infty_{E\shto E_*}$ such that $GF\equiv I$, if and only if
there exists a function $\cP\in H^\infty_{E\shto E}$ whose values
are projections (not necessarily orthogonal) onto $\ran F(z)$ for
all $z\in\D$.
\end{lm}

By this lemma, $T$ is similar to a backward shift if and only if there exists a bounded analytic projection $\cP(z)$ onto $\ran \Theta(z)$. 
Let $\cQ = I-\cP$ be the complementary projection. Then $\cQ(z)^*$ is a projection onto $(\ran\Theta(z))^\perp = \ker \Theta(z)^*$. But as we discussed above in Section \ref{s1.1}, $\ker \Theta(z)^* = E(\bar z)$. 

Notice, that the function $z\mapsto \cQ(z)^*$ is antianalytic, so $T$ is similar to the backward shift iff there exists a bounded antianalytic projection onto $\ker\Theta(z)^*$, or equivalently, a bounded analytic projection onto $E(z) = \ker \Theta(\bar z)^*$.  

Of course this is only a sketch and we leave all the details to the reader as an exercise.

\section{Remark about assumption $\|T\|\le1$}  

In this section we will show that the assumption $\|T\|\le 1$ is essential for Theorem \ref{t0.1}. We will show that if we omit this assumption, it is possible to construct an operator $T$ whose eigenvector bundle is uniformly equivalent to that of $S^*$ (in the scalar Hardy space), but such that $T$ ans $S^*$ 
are not even quazisimilar. 

Let us recall that operators $T_1$ and $T_2$ are called \emph{quazisimilar} if there exist operators $A$ and $B$ with dense ranges and trivial kernels such that 
$$
AT_1 = T_2 A, \qquad T_1 B = B T_2
$$
(if $T_1$ and $T_2$ are similar, then $B=AA^{-1}$). 
\begin{thm}
\label{t5.1}
Given $\e>0$, there exists an operator $T$ such that 
\begin{enumerate}
\item $T$ satisfies assumptions \cond2 through \cond4 from the introduction with $\dim \ker(T-\la I) =1$ for $\la \in \D$;
\item $T^*$ is ``almost isometry'', i.e.,
$$
(1+\e)^{-1} \|x\| \le \|T^*x \| \le (1+\e) \|x\| \qquad \forall x \in H;
$$
\item the eigenvector bundles of $T$ and $S^*$ are almost isomerically equivalent, i.e.,~there exists a bundle map bijection $\Psi$ from the eigenvector bundle of $S^*$ to that of $T$ such that 
$$
(1+\e)^{-1} \|v_\la\| \le \|\Psi(v_\la) \|\le (1+\e) \|v_\la \|
$$
for all $\la\in \D$ and for all $v_\la\in \ker (S^* -\la I)$;
\end{enumerate}
and such that the only operator  $A$ satisfying $AT=S^*A$ is $A=0$, so $A$ is not even quazisimilar to $S^*$. 
\end{thm}


We will construct the operator $T$ as the backward shift, i.e., the adjoint of the forward shift $\bS$ in the space $H^2_w$, with the weight sequence  $w=\{w_k\}_1^\infty$ ($w_k> 0$): 
$$
H^2_{w} :=\left\{f = \sum_{n\ge0} a_n z^n : \|f\|^2_w:= \sum_{n\ge 0} |a_n|^2 w_n <\infty \right\}.  
$$
If one assumes that $\liminf_n |a_n|^{1/n}=1$, the space $H^2_w$ is a space of analytic in the unit disc $\D$ functions. Moreover, for all $\la\in\D$ the functional $f\mapsto f(\la)$ is bounded, so for each $\la\in \D$ there exists a unique function $\bk_\la \in H^2_w$ (the reproducing kernel of $H^2_w$) such that   
\begin{equation}
\label{5.0.5}
\La f, \bk_\la\Ra = f(\la) \qquad \forall f\in H^2_w. 
\end{equation}

The reproducing kernel $\bk_\la$ can be easily computed. Namely, it is easy to see that if $\{ \f_n \}_0^\infty$ is an orthonormal basis in $H^2_w$, then 
$$
\bk_\la (z) = \sum_0^\infty \overline{\f_n(\la)}\vf_n(z) .
$$  
Taking the orthonormal basis $\{z^n/\sqrt{w_n}\}_0^\infty$, we get
\begin{equation}
\label{5.1}
\bk_\la(z) = 	\sum_{n=0}^\infty \frac{z^n}{w_n}. 
\end{equation}

Note that the Hardy space $H^2$ is a particular case ($w_n=1$ for all $n$) of the space $H^2_w$, and formula \eqref{5.1} in this case gives the reproducing kernel $k_\la(z) =1/(1-\bar\la z)$ of $H^2$.

If one assumes that $\sup_n w_{n+1}/w_n<\infty$, then the shift operator $\bS$, $\bS f (z) =z f(z)$ is a bounded operator in $H^2_w$. The adjoint $\bS^*$ is called the \emph{backward shift}, and it is easy to see that 
$$
\bS^*\left(\sum_{n=0}^\infty a_n z^n \right) = \sum_{n=0}^\infty \frac{w_{n+1}}{w_n} a_{n+1} z^n. 
$$
From this formula and \eqref{5.1} one easily concludes that 
$$
\ker(\bS^*-\la I) = \spn\{\bk_{\bar\la}\}. 
$$

It follows from the reproducing kernel property that $\spn\{\bk_\la:\la\in\D\}=H^2_w$, so $T=\bS^*$ satisfies the assumptions \cond1--\cond4 from the Introduction. 

Condition \cond2 of Theorem \ref{t5.1} is satisfied if and only if  
\begin{equation}
\label{5.3}
(1+\e)^{-2}\le w_{n+1}/w_n \le (1+\e)^{2} \qquad \forall n\ge 0. 
\end{equation}

The mapping $\Psi$, $\Psi(ak_\la) = a\bk_\la$, $a\in\C$, $\la\in\D$, is clearly a holomorphic bundle map bijection between the eigenvector bundles of $S^*$ and $\bS^*$. Since $\|k_\la\|^2_{H^2} = \La k_\la, k_\la\Ra = k_\la(\la)\ (= (1-|\la|^2)^{-1} )$ and similarly $\|\bk_\la\|_{H^2_w}^2 = \bk_\la(\la)$, condition \cond3 of Theorem \ref{t5.1} is equivalent to the estimate
\begin{equation}
\label{5.4}
(1+\e)^{-2} k_\la(\la) \le \bk_\la(\la) \le (1+\e)^2 k_\la(\la) \qquad \forall \la \in \D.
\end{equation}
 
\begin{lm}
\label{l5.2}
If $\sup_n w_n=\infty$, then there is no non-zero bounded operator $A$ satisfying $A\bS^*=SA$.  
\end{lm}

\begin{proof} Let  $A\bS^*=S^*A$ for some $A\ne 0$. Then  $\bS A^* = A^*S$ and therefore $\bS^n A^* = A^* S^n$.  

Take $f\in H^2$ such that $A^*f=\sum_0^\infty a_n z^n \ne0$.  Pick $m$ such that  $a_m\ne 0$. Then  
$$
\| \bS^n A f\|^2 = \sum_{j=0}^\infty |a_j|^2 w_{j+n} \ge |a_m|^2 w_{m+n}, 
$$
so $\sup_n \|\bS^n Af\| = \infty$ because $\sup_n w_{m+n}=\infty$. 

On the other hand,  $\|A^*S^n f\|_{H^2_w} \leq \|A^*\|\|f\|_{H^2}$, giving us a contradiction.
\end{proof}

So, to prove the theorem, we need to find an unbounded sequence $\{w_n\}_1^\infty$ satisfying \eqref{5.3}, and such that  \eqref{5.4} holds. 
We define the sequence $\{w_n\}_1^\infty$ to be $1$ for all $n$ except  in  sparse intervals from $N_j$ to $N_j+2j$; the numbers $N_j$ will  be specified later. On  the intervals $[N_j, N_j+2j]$ the sequence has ``spikes'': $\ln w_n$ on $[N_j, N_j+2j]$  is the piecewise affine function with slope $\pm2\ln(1+\e)$, increasing from $0$ to $2j\ln(1+\e)$ on $[N_j, N_j+j]$ and decreasing back to $0$ on $[N_j+j, N_j+2j]$; see Figure \ref{fig1}.

Formally, we  write
$$
\ln w_n=
\begin{cases} 2m \ln(1+\epsilon) &n=N_k+m,\  0 \leq m \leq j,
\\
2(j-m)\ln(1+\epsilon) & n=N_k+j+m, \ 0 \leq m \leq j, 
\\
0 & \text{ otherwise. }
\end{cases}
$$
We pick a sequence $N_j$ such that 
\begin{align*}
N_j+2j  & < N_{j+1} \\
\frac{2j-1}{N_j+2j} & \leq \frac{\alpha}{2^j} \,, 
\end{align*}
where $\alpha$ is a small number such that $1-\alpha\ge(1+\e)^{-2}$. 

The sequence $\{w_n\}_0^\infty$ is clearly unbounded. Because of the slope condition for $\ln w_n$ condition \eqref{5.3} is satisfied.

\setlength{\unitlength}{0.6pt}
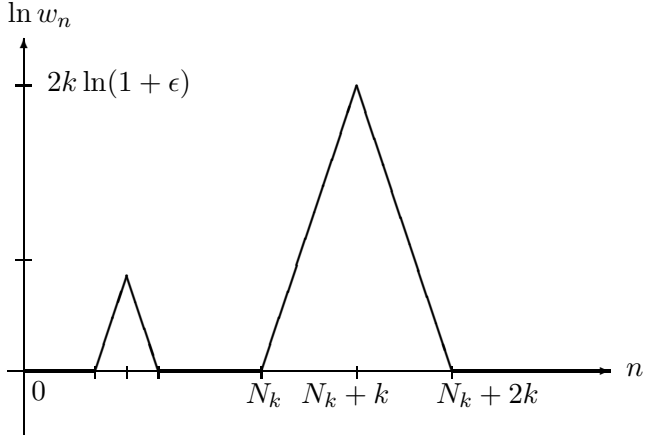
\begin{figure}
\begin{center}
\begin{picture} (350,300)

\put (-10,32){$0$}
\put (-0, 227){$2k\ln(1+\epsilon)$}
\put (-15,10){\vector(0,1){250}}
\put(-25,50){\vector(1,0){380}}
\put (-25, 270){$\ln w_n$}
\put (365,47){$n$}

\put(30,45){\line(0,1){8}}
\put(50,45){\line(0,1){8}}
\put(70,45){\line(0,1){8}}

\put(135,45){\line(0,1){8}}
\put(195,45){\line(0,1){8}}
\put(255,45){\line(0,1){8}}
\put(125,30){$N_{k}$}
\put(160, 30){$N_{k}+k$}
\put(245,30){$N_{k}+2k$}

\put(-20,120){\line(1,0){10}}
\put(-20,230){\line(1,0){10}}


\thicklines
\put(-15,50){\line(1,0){45}}

\thicklines
\put(30,50){\line(1,3){20}}

\thicklines
\put(50,110){\line(1,-3){20}}

\thicklines
\put(70,50){\line(1,0){65}}

\thicklines
\put(135,50){\line(1, 3){60}}

\thicklines
\put(195,230){\line(1,-3){60}}

\thicklines
\put(255,50){\line(1,0){100}}

\end{picture}
\end{center}
\caption{The function $\ln w_n$: two ``spikes'' are shown}%
{\protect\label{fig1}}%
\end{figure}

The Lemma below completes the proof of the theorem.

\begin{lm}
\label{l5.3}
For the weight $w=\{w_n\}_0^\infty$ constructed above, the reproducing kernel $\bk_\la$ of $H^2_w$ satisfies the inequality \eqref{5.4}.
\end{lm}

\begin{proof}[Proof of Lemma \ref{5.3}]
Since $w_n\ge 1$, 
$$
\bk_\la(\lambda)=\sum_{n  \geq 0} \frac{1}{w_n} |\lambda|^{2n} \leq \sum_{n \geq 0}|\lambda|^{2n} = k_\la(\la) \le (1+\e)^2 k_\la(\la), 
$$ 
so one estimate is obvious. To get the other one, it is enough to show that $k_\la(\la)-\bk_\la(\la)\le \alpha k_\la(\la)$. 

Since $1-1/w_n\ne0$ only for $n\in (N_j, N_j+2j)$, we can write 
$$
k_\la(\la)-\bk_\la(\la) = \sum_{n \geq 0} \left(1-\frac{1}{w_n}\right)|\lambda|^{2n} =
\sum_{j=1}^\infty \sum_{n=N_j+1}^{N_1+2j-1} \left(1-\frac{1}{w_n}\right)|\lambda|^{2n}.
$$
For each $j$
\begin{align*}
\sum_{n=N_j+1}^{N_j+2j-1} \left(1-\frac{1}{w_n}\right)|\lambda|^{2n} & \le 
\sum_{n=N_j+1}^{N_j+2j-1} |\lambda|^{2n} 
\\
& = \frac{|\la|^{2(N_j+1)} - |\la|^{2(N_j+2j)}}{1-|\la|^2} \le \frac{A_j}{1-|\la|^2},
\end{align*}
where
$$
A_j = \max\{ x^{N_j+1} - x^{N_j+2j}: 0\le x\le 1\}  
$$
(here $x=|\la|^2$). The maximum is attained at $x=\left(\frac{N_j+1}{N_k+2j}\right)^{\frac{1}{2j-1}}$, and 
$$
A_j = \left(\frac{N_j+1}{N_j+2j}\right)^{\frac{N_j+1}{2j-1}}\frac{2j-1}{N _j+2j} \le \frac{2j-1}{N _j+2j} \le \frac{\alpha}{2^j}
$$
by our choice of $N_j$. 

Summing over $j$ we get 
$$
k_\la(\la)-\bk_\la(\la) \le \frac{1}{1-|\la|^2}\sum_{j=1}^\infty \alpha 2^{-j} = \frac{\alpha}{1-|\la|^2} =\alpha k_\la(\la),
$$
so the lemma is proved. 
\end{proof}

\def\cprime{$'$}
\providecommand{\bysame}{\leavevmode\hbox
to3em{\hrulefill}\thinspace}

\end{document}